\documentclass[letter,11pt,twopage]{amsart}

\usepackage{graphicx}
\usepackage{graphics}
\usepackage{hyperref}
\usepackage{xifthen}

\usepackage[inline]{enumitem}
\setenumerate{leftmargin=*}

\usepackage{amsmath,amssymb}
\usepackage{amsthm}
\usepackage{tikz-cd}
\usepackage{braket}
\usepackage{mathtools}
\usepackage{mathrsfs}

\theoremstyle{plain}
\newtheorem{theorem}{Theorem}[section]

\newtheorem{lemma}[theorem]{Lemma}
\newtheorem*{lemma*}{Lemma}
\newtheorem{proposition}[theorem]{Proposition}
\newtheorem{corollary}[theorem]{Corollary}
\newtheorem*{corollary*}{Corollary}

\newtheorem{claim}[theorem]{Claim}
\newtheorem*{claim*}{Claim}
\newtheorem*{theorem*}{Theorem}

\newtheoremstyle{break}%
{}{}%
{\itshape}{}%
{\bfseries}{}
{\newline}{}

\theoremstyle{break}

\theoremstyle{definition}
\newtheorem{definition}[theorem]{Definition}

\newtheorem*{notation}{Notation}

\theoremstyle{remark}

\newtheorem{obs*}[theorem]{Observation}
\newtheorem*{remark}{Remark}
\newtheorem{remark*}[theorem]{Remark}

\newcommand{\N}{\mathbb{N}}

\newcommand{\R}{\mathbb{R}}

\newcommand{\itembf}[2]{%
\ifthenelse{\isempty{#1}}{\item {\bfseries #2}}{\item[\textbf{#1}]{\bfseries #2}}%
}

\newcommand{\dfn}{\mathrel{\mathop:}=}

\newcommand{\sqed}{\hfill $/\!\!/$}

\newcommand{\scratch}[1]{{\color{red}\textbf{#1}}}

\DeclareMathAlphabet{\mathbbold}{U}{bbold}{m}{n}
\DeclareMathAlphabet{\cmb}{OT1}{cmbr}{m}{n}

\DeclareMathOperator{\img}{im}

\DeclareMathOperator{\Map}{Map}

\DeclareMathOperator{\Co}{Co}

\DeclareMathOperator{\inter}{int}

\DeclareMathOperator{\diam}{diam}

\DeclareMathOperator{\asdim}{asdim}

\newlist{clist}{enumerate}{1}
\setlist[clist]{label=(\roman*),itemsep=0ex,
  topsep=.5ex,parsep=0ex,leftmargin=3.5em}

\newlist{citem}{itemize}{1}
\setlist[citem,1]{label=\textbullet,topsep=1ex,itemsep=1ex,parsep=0ex}

\ifdefined\clean
\renewcommand{\scratch}[1]{{}}
\else
\fi

\usepackage{amsaddr}
\usepackage[normalem]{ulem}

\newcommand{\A}{\mathcal{A}}
\newcommand{\G}{\mathcal{G}}
\newcommand{\LL}{\mathcal{L}}
\DeclareMathOperator{\el}{\mathcal{EL}}
\newcommand{\M}{\mathcal{M}}
\DeclareMathOperator{\SC}{\mathcal{S}}
\DeclareMathOperator{\base}{base}
\DeclareMathOperator{\trans}{trans}
\newcommand{\cc}{\mathcal{C}}
\newcommand{\X}{\mathscr{X}}

\DeclareMathOperator{\PMap}{PMap}
\newcommand{\K}{\mathcal{K}}

\begin{document}

\title
{The asymptotic dimension of the grand arc graph is infinite}
\author{Michael C.\ Kopreski}


\begin{abstract}
  Let $\Sigma$ be a compact, orientable surface of genus $g$, and let
  $\Gamma$ be a relation on $\pi_0(\partial \Sigma)$ such that the
  prescribed arc graph $\A(\Sigma,\Gamma)$ is Gromov-hyperbolic and
  non-trivial.  We show that $\asdim \A(\Sigma,\Gamma) \geq
  -\chi(\Sigma) - 1$, from which we prove 
  that the asymptotic dimension of the grand arc graph is infinite.
  More generally, an \textit{arc and curve model} on
  $\Sigma$ is a graph of 
  simple arc and curves on $\Sigma$, on which 
  $\PMap(\Sigma)$ acts by permuting vertices.  
  We prove that any connected, Gromov-hyperbolic
  cocompact arc and curve model $\M$ has $\asdim \M \geq g - \lceil\frac{1}{2}
  \chi(\Sigma)\rceil$, and that a broad class of arc and curve
  models on infinite-type surfaces has infinite
  asymptotic dimension.
\end{abstract}

\maketitle

\section{Introduction}

Let $\Sigma$ be a compact, orientable surface with boundary, and let
$\Gamma$ be a relation on $\pi_0(\partial \Sigma)$. 
A simple, essential arc $a$ in $\Sigma$ 
is \textit{$\Gamma$-allowed} if it joins boundary
components in $\Gamma$.  The
\textit{$\Gamma$-prescribed arc graph} $\A(\Sigma,\Gamma)$ is the
full subgraph of $\A(\Sigma)$ spanned by isotopy classes of
$\Gamma$-allowed arcs. 
We assume throughout that
$\A(\Sigma,\Gamma)$ is \textit{non-trivial}, \textit{i.e.}\ 
$\chi(\Sigma) \leq -1$,
$\Sigma \neq \Sigma_0^3$, and $\Gamma \neq \varnothing$.

We suppose $\A(\Sigma,\Gamma)$ is
$\delta$-hyperbolic. 
If $\Sigma = \Sigma_0^4$, then $\A(\Sigma,\Gamma) \subset
\A(\Sigma_0^4)$ is a quasi-tree
and $\asdim \A(\Sigma,\Gamma) = 1$.  Otherwise, we prove a lower
bound:
\begin{theorem}\label{thm:asdim}
  If $\A(\Sigma,\Gamma)$ is $\delta$-hyperbolic, then $
-\chi(\Sigma) - 1 \leq \asdim \A(\Sigma,\Gamma).
$
\end{theorem}

\noindent
For $\Omega$ an infinite-type surface with finite
grand splitting, 
let $\G(\Omega)$ denote the \textit{grand arc graph} on $\Omega$
\cite{grand}.  By applying Theorem~\ref{thm:asdim}, we obtain the 
following:
\begin{theorem}\label{thm:ga}
  If $\G(\Omega)$ is non-empty and connected, then 
  $\asdim \G(\Omega) = \infty$.
\end{theorem}

More generally, 
an \textit{arc and curve model} on 
a surface $\Omega$ is a connected graph whose vertices are collections of
(possibly intersecting) simple arcs and curves, with an action of
$\PMap_c(\Omega)$ induced by the permutation of its vertices.  
An arc and curve model $\M$ is \textit{witness-cocompact} if it has a
(compact) witness and in each witness has uniformly bounded
geometric (self-)intersection over the vertices and edges in $\M$.

\begin{theorem}\label{thm:ma_asdim}
  If $\Omega$ is compact and $\M$ is $\delta$-hyperbolic, then
  $\asdim \M \geq g(\Omega) - 
  \lceil\frac 12 \chi(\Omega)\rceil$. If 
  $\Omega$ is infinite-type, then $\asdim \M = \infty$.
\end{theorem}

In Section~\ref{sec:asdim}, we use the theory of
alignment-preserving maps \cite{align} to show that the Gromov
boundary
$\partial \A(\Sigma,\Gamma)$ contains $\partial \A(\Sigma)$. 
From results of Gabai \cite{gabai} and Schleimer \cite{phoon}  
we obtain a compact subspace $Z \subset \partial \A(\Sigma,\Gamma)$ 
of dimension $-\chi(\Sigma)-2$. We then prove that 
$\asdim \A(\Sigma,\Gamma)
\geq \dim Z + 1$, extending a result for proper $\delta$-hyperbolic
spaces, whence Theorem~\ref{thm:asdim} follows.

In Section~\ref{sec:ga}, we show that witness subsurfaces $W \subset
\Omega$ for $\G(\Omega)$ of
arbitrarily large complexity admit prescribing relations $\Gamma$
such that $\A(W,\Gamma)$ quasi-isometrically embeds into
$\G(\Omega)$, where $\Omega$ is an infinite-type surface with 
finite grand splitting.  In fact, $W$ may be chosen so that either
$\A(W,\Gamma)$ has large coarse rank or it is
$\delta$-hyperbolic: Theorem~\ref{thm:ga} thus follows from
Theorem~\ref{thm:asdim} and the monotonicity of asymptotic
dimension.

Section~\ref{sec:cm} generalizes the techniques in
Sections~\ref{sec:asdim} and \ref{sec:ga} to witness-cocompact
arc and curve models,
which include prescribed arc graphs, the grand arc graph, 
the marking complex, 
and many other multiarc and curve graphs. 
In addition to tools developed in Section~\ref{sec:asdim}, 
we utilize properties of the hierarchically hyperbolic
structure of such graphs in the finite-type setting \cite{hhs_gen}.
Theorem~\ref{thm:ma_asdim} follows analogously in the witness-cocompact
case.

\begin{remark}
  For the reader interested in only Theorem~\ref{thm:ma_asdim}
  (which does imply Theorem~\ref{thm:ga} and a weaker version of
  Theorem~\ref{thm:asdim}, albeit with more technology
  than necessary), it suffices to read 
  Sections~\ref{sec:lowerbound} and \ref{sec:cm}.
\end{remark}

\subsection{Background} An orientable
surface $\Omega$ has \textit{infinite
topological type} if its fundamental group is not finitely
generated, or equivalently 
if $\inter(\Omega)$
has infinite genus or infinitely many punctures (we typically assume
$\partial \Omega = \varnothing$). Beginning with
a 2009 blog post of Calegari \cite{danny},
mapping class groups of infinite-type surfaces have been 
objects of considerable contemporary study: see \cite{avbig,
problemlist} for surveys of recent results and open problems.

An infinite-type surface $\Omega$ 
is classified by its genus and
\textit{end space}, which is obtained as the inverse limit of the
complementary components of a compact exhaustion \cite{richards}; 
its mapping class group $\Map(\Omega)$ is a non-compactly
generated Polish group.  Given mild assumptions, Mann and 
Rafi \cite{mannrafi}
classify when $\Map(\Omega)$ admits a generating set that is
\textit{coarsely bounded (CB)}, or 
bounded in any left-invariant metric, and hence a well defined
quasi-isometry type in the sense of \cite{rosendal}. 
Mann--Rafi also define a preorder on the
ends of $\Omega$ corresponding to topological complexity.  
We denote by $\mathscr{M}(\Omega)$ the non-empty subspace of maximal
ends with respect to this preorder. 

When $\Map(\Omega)$ is locally CB (and in particular when
it is CB-generated), Bar-Natan and Verberne define the \textit{grand
splitting} $\SC(\Omega)$, a canonical and $\Map(\Omega)$-invariant
partition of $\mathscr{M}(\Omega)$  
into finitely many disjoint sets $E_i \in \SC(\Omega)$, 
each of which is either a singleton or Cantor set.
 A \textit{grand arc} in $\Omega$
is a bi-infinite simple 
arc converging to ends in distinct sets in the
grand splitting \cite{grand}.
\begin{definition}[Bar-Natan--Verberne]\label{def:ga}
  Let $\Omega$ be an infinite-type surface. 
  The \textit{grand arc graph}
  $\G(\Omega)$ is the simplicial graph with vertices corresponding 
  to isotopy classes of grand arcs and edges determined by
disjointness.  
\end{definition}

The grand arc graph $\G(\Omega)$ is an arc and curve model
for $\Omega$ which generalizes the ray graph defined by 
Calegari \cite{danny} on $S^2 \setminus \text{Cantor set}$ and 
for surfaces with stable endspace extends
the omnipresent arc graph defined by Fanoni--Ghaswala--McLeay
\cite{fgm}.  $\Map(\Omega)$ acts naturally on $\G(\Omega)$ by
isometries.  Bar-Natan--Verberne classify the
$\delta$-hyperbolicity of $\G(\Omega)$ and show that
when $\G(\Omega)$ is $\delta$-hyperbolic, the action of 
$\Map(\Sigma)$ is quasi-continuous, extends continuously to
$\partial \G(\Omega)$, and has loxodromic elements.  

\begin{notation}
  We typically denote by $\Sigma$ a compact, orientable surface,
and by $\Omega$ an arbitrary orientable surface that may have either
finite or infinite topological type. 
\end{notation}

\subsubsection{Prescribed arc graphs, witnesses}
Prescribed arc graphs were defined by the author in \cite{pac}
as arc and curve models of finite-type surfaces that
quasi-isometrically embed into $\G(\Omega)$.  Excepting trivial
cases they are connected and infinite-diameter and  
their $\delta$-hyperbolicity is fully
determined by the prescribing relation $\Gamma$:

\begin{theorem}[{\cite[Thm.\ 1.3]{pac}}]\label{thm:pac_hyp}
Assume that $\A(\Sigma,\Gamma)$ is non-trivial.
Then if $\chi(\Sigma) \geq -2$ or
$\Sigma = \Sigma_0^{n+1}$ and $\Gamma$ is a $n$-pointed star then
$\A(\Sigma,\Gamma)$ is $\delta$-hyperbolic.  Otherwise,
$\A(\Sigma,\Gamma)$ is (uniformly) $\delta$-hyperbolic if and only
if $\Gamma$ is not bipartite.
\end{theorem}

We note that if $\Gamma \subset \Gamma'$ then every
$\Gamma$-allowed arc is $\Gamma'$-allowed, which induces a
simplicial map $\iota : \A(\Sigma,\Gamma) \to
\A(\Sigma,\Gamma')$.  This map is $3$-coarsely surjective 
\cite[Lem.~2.10]{pac}.  In particular, since the prescribed arc
graph with the complete relation is exactly $\A(\Sigma)$,
$\A(\Sigma,\Gamma)$ always coarsely surjects onto $\A(\Sigma)$.

A compact, essential ($\pi_1$-injective,
non-peripheral) subsurface without pants components 
is a \textit{witness} for a given arc and curve model if 
each component intersects every vertex. We call a
witness $W\subset \Sigma$ for $\A(\Sigma,\Gamma)$ a
\textit{$\Gamma$-witness}.

\subsubsection{Boundaries of non-proper $\delta$-hyperbolic spaces}  
In general, if $\A(\Sigma,\Gamma)$ is non-trivial then it is
non-proper, and likewise for any admissible arc and curve model 
with sufficient complexity.
For a geodesic $\delta$-hyperbolic space $X$, by 
$\partial X$ we always mean the sequential boundary of $X$;  
when $X$ is non-proper, $\partial X$ may be non-compact.
In this setting, $\partial X$ does not coincide with the geodesic
boundary, but is instead homeomorphic to the quasi-geodesic boundary 
\cite{yo}.  We will make use of the following
statement by Hasegawa, from a construction of
Kapovich--Benakli \cite[Rmk.~2.16]{kb}: 

\begin{remark*}[{\cite[Prop.~4]{yo}}]\label{rmk:qgr} 
  Fixing $x_0 \in X$, for any $z \in \partial X$ there exists a
  $(1+4\delta,12\delta)$-quasi-geodesic ray $\rho : [0,\infty) \to
  X$ based at $x_0$ with $[\rho(n)] = z$.
\end{remark*}

Any quasi-isometry between geodesic
 $\delta$-hyperbolic spaces $X \to Y$ extends to a map $X \cup
\partial X \to Y \cup \partial Y$ that restricts to
a homeomorphism on the
boundaries (\textit{e.g.}\ applying the proof of 
\cite[Thm.~11.108]{dk}).
Given $x,y \in X \cup \partial X$, let $(x|y)_{x_0}$ denote their
Gromov product at $x_0$.  We occassionally 
omit the basepoint, which is changeable up to bounded error. 

\subsubsection{Ending laminations} 
Let $\chi(\Sigma) \leq -1$, hence 
fix a (finite-area) hyperbolic metric
for $\Sigma$ with geodesic boundary.  
We recall that a \textit{geodesic lamination} on
$\Sigma$ is a closed subset $L
\subset \Sigma$ which decomposes (in fact, uniquely) into
pair-wise disjoint simple geodesic leaves. 
$L$ is \textit{minimal}
if it has no proper sublaminations, or equivalently, 
if every leaf is dense in $L$.  

\begin{definition}\label{def:filling}
  Given a connected subspace $X \subset \Sigma$ with non-trivial
  $\pi_1$-image, if $Y \subset \Sigma$ is 
  the smallest essential subsurface 
  containing $X$ up to isotopy,  then $Y$ is
  \textit{filled} by $X$.  If $Y = \Sigma$, then $X$ is
  \textit{filling}.
\end{definition}

\begin{definition}\label{def:el0}
   The space of \textit{ending laminations} $\el(\Sigma)$ 
   is the set of filling minimal laminations 
  on $\Sigma$, equipped with the coarse Hausdorff topology.  
  Similarly, let $\el_0(\Sigma)$ 
  denote the space of minimal laminations that fill a subsurface
  containing $\partial \Sigma$, again with the coarse
  Hausdorff topology.
\end{definition}

$\el(\Sigma)$ and $\el_0(\Sigma)$ give explicit
descriptions for the hyperbolic boundaries of $\cc(\Sigma)$ and
$\A(\Sigma)$, respectively (see \cite{klarreich} and \cite{phoon}):

\begin{theorem}[Klarreich, Schleimer]\label{thm:el0bdry}
  $\el(\Sigma) \cong \partial \cc(\Sigma)$ and 
  $\el_0(\Sigma) \cong \partial \mathcal{A}(\Sigma)$.
\end{theorem}

\subsubsection{Markings} In Section~\ref{sec:cm}, we will make use
of \textit{markings} on surfaces in the sense of \cite{MMII}.  
For an essential simple closed curve $a \subset \Omega$, let
$\cc(a)$ denote the curve graph of the annulus with core $a$ and
$\pi_a$ the corresponding (set-valued) subsurface projection. 

\begin{definition}\label{def:marking}
  A \textit{marking}
  $\mu = \{(a_i,t_i)\}$ on a surface $\Omega$ is an essential
  simple multicurve $\{a_i\}$, denoted $\base \mu$, along with a
  collection of (possibly empty) diameter $1$ subsets $t_i \subset
  \cc (a_i)$; for $a_i \in \base \mu$, let
  $\trans_\mu(a_i) = t_i$ denote the associated 
  \textit{transversal.}
\end{definition}

\noindent
A marking $\mu$ is \textit{complete} 
if $\base \mu$ is a pants decomposition
and every transversal is non-empty.  If $\mu$ is complete and 
for each component $(a,t) \in \mu$ $t = \pi_a b$ for some simple
closed curve $b \neq a$ disjoint from $\base \mu \setminus \{a\}$ 
that intersects $a$ minimally, then $\mu$ is \textit{clean}.

\begin{definition}[{\cite[Def.~2.2]{hhs_gen}}]\label{def:lcmarking}
  A marking $\mu$ on a surface $\Omega$ is \textit{locally clean} if the maximal
  submarking $\mu'$ with only non-empty transversals is complete and
  clean in each component of $\Omega \setminus (\mu \setminus \mu')$
  that it intersects.
\end{definition}

Let $\Delta \subset \Omega$ be an essential, non-pants
subsurface.  Like multicurves, markings have a 
subsurface projection $\pi_\Delta(\mu) \subset \cc (\Delta)$.
If $\Delta$ is an annulus parallel to some $a \in \base
\mu$, then $\pi_\Delta(\mu) \dfn \trans_\mu(a) \subset \cc(\Delta)$.
Otherwise, $\pi_\Delta(\mu) \dfn \pi_\Delta(\base \mu)$. We say 
$\Delta$ intersects $\mu$ if and only if $\pi_\Delta(\mu) \neq
\varnothing$.  For an
essential simple closed curve $c \subset \Omega$, again let 
$\pi_c(\mu)$ denote the projection to the annulus with core $c$.

\begin{definition}\label{def:geom_int}
  Let $\mu,\nu$ be two markings on
  $\Omega$.  Then their \textit{geometric
  intersection number} $i(\mu,\nu)$ is defined as follows:
  \[
i(\mu,\nu) \dfn i(\base\mu,\base\nu) + \sum_{a \in \base \mu \cup
\base\nu} \diam_{\cc (a)}(\pi_a \mu \cup \pi_a \nu)
  \]
\end{definition}

\subsubsection{Alignment-preserving maps} We briefly recall 
the theory of
alignment-preserving maps from \cite{align}.
Let $X$ be a geodesic metric space.
Then a triple 
$(x,y,z) \in X^3$ is \textit{$K$-aligned} if $d(x,y) + d(y,z) \leq
d(x,z) + K$. A Lipschitz map between geodesic metric spaces 
$f : X \to Y$ is \textit{coarsely alignment preserving} if
there exists $K \geq 0$ for which $f$ 
maps any $0$-aligned triple in $X$ to a $K$-aligned triple in $Y$.

Suppose that $f : X \to Y$ is a coarsely alignment preserving map
between geodesic 
$\delta$-hyperbolic spaces.  Then we define $\partial_Y X
\subset \partial X$ to be 
\[ \partial_Y X \dfn \{[\gamma] \in \partial X \;\vline\;
\gamma : \R^+ \to X \text{ quasi-geodesic, } 
\diam_Y (f\gamma (\R^+)) = \infty \}.\]

\begin{theorem}[Dowdall--Taylor, {\cite[Thm.~3.2]{align}}]
  \label{thm:dtboundaries}
  Let $f : X \to Y$ be a coarsely surjective, coarsely alignment
  preserving map between geodesic 
  $\delta$-hyperbolic spaces.  Then $f$
  admits an extension to a homeomorphism $\partial f : \partial_Y X
  \to \partial Y$ such that if $x_n \to \omega \in
  \partial_Y X$, then $f(x_n) \to \partial f(\omega)$.
\end{theorem}

\section{Asymptotic dimension of
$\A(\Sigma,\Gamma)$}\label{sec:asdim} 

When $\Sigma = \Sigma_0^4$, then $\A(\Sigma,\Gamma) \subset
\A(\Sigma_0^4)$ is an infinite-diameter connected
subgraph of a quasi-tree, hence likewise a quasi-tree: $\asdim
\A(\Sigma,\Gamma) = 1$.  For $\Sigma \neq \Sigma_0^4$, 
we first prove $\el_0(\Sigma) \cong \partial \A(\Sigma) \subset
\partial \A(\Sigma,\Gamma)$ for $\Gamma$ not bipartite.  

\begin{lemma}\label{lem:iota_ap}
  If $\Gamma$ is not bipartite
  then for any $\Gamma' \supset \Gamma$ the induced coarse
  surjection $\iota : \A(\Sigma,\Gamma) \to \A(\Sigma,\Gamma')$ is
  uniformly coarsely alignment-preserving.
\end{lemma}

\begin{proof}
  We first claim that if $\Gamma$ is not bipartite, then geodesics
  in $\A(\Sigma,\Gamma)$ are uniformly (independent of $\Gamma$)
  Hausdorff close to unicorn paths with coarsely the same endpoints,
  and \textit{vice versa}.  If $\Gamma$ is not bipartite and if
  $\Sigma = \Sigma_1^2$ then $\Gamma$ is not two loops, then the
  claim holds by \cite[\S3]{pac}. If instead $\Sigma=\Sigma_1^2$
  and $\Gamma = \ell_1 \cup \ell_2$ is two loops, then $\iota :
  \A(\Sigma,\ell_1) \to \A(\Sigma,\Gamma)$ is a quasi-isometry by
  \cite[Lem.~5.2]{pac} and we apply the Morse lemma. We observe that
  if $\Gamma$ is not bipartite then neither is $\Gamma'$. 


  We claim that for any
  geodesic $\gamma$ between $a,b\in\A(\Sigma,\Gamma)$, 
  $\iota \gamma$ is
  uniformly Hausdorff close to a geodesic between $\iota(a),
  \iota(b)$, whence the proof follows.
 Let $\gamma'$ be a unicorn path close to $\gamma$, in the sense
 above.  $\iota$ is
 Lipschitz, hence  $\iota \gamma, \iota \gamma'$ are close; since
$ \iota\gamma'$ is a unicorn path in $\A(\Sigma,\Gamma')$, choose
a geodesic $\gamma''$ close to $\iota \gamma'$.  By the Morse lemma,
there exists a
geodesic $\gamma'''$ between $\iota(a),\iota(b)$ that is close to
$\gamma''$, hence close to $\iota\gamma$.
\end{proof}

\noindent
Applying Theorem~\ref{thm:dtboundaries}, 
we obtain the desired embedding. 

\begin{corollary}\label{cor:iota_boundary}
  If $\Gamma$ is not bipartite and $\Gamma' \supset \Gamma$, then
  there exists an embedding $(\partial \iota)^{-1} : \partial
  \A(\Sigma,\Gamma') \to \partial \A(\Sigma,\Gamma)$. \qed
\end{corollary}

By \cite[\S 5]{pac}, if $\Sigma \neq \Sigma_0^4$ and
$\A(\Gamma,\Sigma)$ is $\delta$-hyperbolic, then
\begin{enumerate*}[label=(\roman*)]
  \item $\Sigma = \Sigma_0^{n+1}$ and $\Gamma$ is an $n$-pointed
    star, 
  \item $\Sigma = \Sigma_1^2$ and $\Gamma$ is a
    non-loop edge, or
  \item $\Gamma$ is not bipartite.
\end{enumerate*}
In case (i), by \cite[Lem.~5.4]{pac} $\A(\Sigma,\Gamma)$ is
quasi-isometric to $\A(\Sigma,\ell_0)$, where $\ell_0$ is a single
loop and hence not bipartite.  Thus for cases (i) and (iii),
Corollary~\ref{cor:iota_boundary} implies
$\partial\A(\Sigma)
\subset \partial \A(\Sigma,\Gamma)$. 
In case (ii), every $\Gamma$-witness is in fact a
witness for the usual arc graph: by \cite{hhs_gen}
$\A(\Sigma,\Gamma)$ and $\A(\Sigma)$ have the same quasi-isometry
type, hence $\partial \A(\Sigma,\Gamma) \cong \partial
\A(\Sigma)$.

\begin{proposition}\label{prop:boundary}
  Let $\Sigma \neq \Sigma_0^4$ and $\A(\Sigma,\Gamma)$ be
  $\delta$-hyperbolic.  
  $\partial \A(\Sigma) \cong \el_0(\Sigma)$ embeds canonically into
  $\partial \A(\Sigma,\Gamma)$. \qed
\end{proposition}

\subsection{A lower bound}\label{sec:lowerbound} 

From \cite{gabai}, we have the following:
\begin{theorem}[Gabai]\label{thm:gabai}
   Let $S$ be the $(n + 4)$-times punctured sphere for $n \geq 0$.  
   Then $\el(S)$ is
   homeomorphic to the $n$-dimensional N\"obeling space.
\end{theorem}

For any $\Sigma$ with $\chi(\Sigma) \leq -2$, let $n = n(\Sigma) =
-\chi(\Sigma) - 2$ and let $\Gamma$ be a prescribing relation such
that $\A(\Sigma,\Gamma)$ is $\delta$-hyperbolic.   
We may choose an essential $(n+4)$-punctured
sphere  $S$ that contains all of the punctures of $\Sigma$, 
thus $\el(S) \subset \el_0(\Sigma) \cong \partial \A(\Sigma)$.
Then applying Proposition~\ref{prop:boundary} and
Theorem~\ref{thm:gabai}, 
$\partial \A(\Sigma,\Gamma)$ contains the $n$-dimension
N\"obeling space, and in particular,  
a compact subspace $Z \subset
\el(S)$ of topological 
dimension $n$ by the universal embedding property of
N\"obeling spaces \cite{nobeling}.  

For the remainder of the section, we will prove the 
following generalization 
of a result for proper $\delta$-hyperbolic spaces
(\textit{e.g.}\ \cite[Prop.~6.2]{bleb}):

\begin{proposition}\label{prop:bdry_cpt}
  Let $X$ be a geodesic $\delta$-hyperbolic space with $Z \subset
  \partial X$ compact.  Then $\asdim X \geq \dim Z + 1$.
\end{proposition}

\noindent
Since $\partial \A(\Sigma,\Gamma)$ contains a
$n(\Sigma)$-dimensional compact subspace for $\chi(\Sigma) \leq -2$,
Theorem~\ref{thm:asdim} follows (vacuously for $\chi(\Sigma) > -2$).

By $\delta$-hyperbolic, we mean that geodesic triangles are
$\delta$-slim.
Let $X$ be a geodesic $\delta$-hyperbolic space and let $Z \subset
\partial X$ be compact. 
A metric $d : \partial X \times \partial X \to
[0,\infty)$ is \textit{visual} if there exist $k_1,k_2$ and $a > 0$
such that \[
  k_1 a^{-(\xi | \xi')} \leq d(\xi,\xi') \leq k_2 a^{-(\xi|\xi')}.
\]
Such metrics always exist \cite[Prop.\ III.H.3.21]{bh} and are
compatible with the usual topology on the (sequential) boundary:
$d(\xi_i,\xi) \to 0$ if and only if $(\xi_i | \xi) \to \infty$,
which is equivalent to $\xi_i \to \xi$. 

\begin{notation}
  Where unambiguous, we denote by $|xx'|$ the distance between $x,x'
  \in X$ a metric space.  Given a specified 
  basepoint $o \in X$, let $|x| \dfn |ox|$.
\end{notation}

\begin{definition}\label{def:cone}
  For $(Z,d)$ a bounded metric space, the \textit{hyperbolic cone} 
  $\Co Z$ 
  is the topological cone 
  $Z \times [0,\infty) / Z \times \{0\}$ endowed with
  the following metric. Let $\mu = \pi / \diam(Z)$.  
  For any $x = (z,t),x' = (z',t') \in \Co Z$, consider a geodesic
  triangle $\bar{o}\bar{x}\bar{x}' \subset H^2$ 
  with $|\bar o \bar x| = t, |\bar o \bar x'| =
  t'$, and $\angle_{\bar o}(\bar x,\bar x') =
  \mu|zz'|$.  Then let $|xx'| \dfn |\bar x\bar x'|$.
\end{definition}

This metric is compatible with the usual 
topology on $\Co Z$.  In addition, $\Co Z$ is $\delta$-hyperbolic,
$Z \hookrightarrow \partial \Co Z$ via the geodesic rays
$\gamma_z : t \mapsto (z,t)$, and $d$ is visual for 
$Z\subset \partial \Co Z$ with respect to $\Co Z$ 
\cite[Prop.~6.1]{buyalo}.   We fix $o = Z \times
\{0\}$ as a basepoint for $\Co Z$.
Analogously to \cite[Prop.~6.2]{buyalo}, we have the following:

\begin{lemma}\label{lem:cone_embed}
  Let $X$ be a geodesic $\delta$-hyperbolic space and let 
  $Z \subset \partial X$ be compact.  Then $\Co Z$
  quasi-isometrically embeds into $X$.
\end{lemma}

\begin{proof}
  Fix a basepoint $x_0 \in X$ and let $\delta' =
  \delta(\Co Z)$. 
  Since $d$ is visual for both $X$ and $\Co Z$, 
  up to rescaling $X$ we may assume that $(z|z')_{x_0}$ and
  $(z|z')_o$ are uniformly close for all $z, z' \in Z$.
  For each $z \in Z$, 
  fix a representative $(\kappa_0,12\delta)$-quasi-geodesic ray 
  $\rho_z \in z$ eminating from $x_0$ by Remark~\ref{rmk:qgr}, where
  $\kappa_0 = 1 + 4\delta$.
  Let $\iota : \Co Z \to X$ be the map $(z,t) \mapsto \rho_z(t)$.  
  
  Since $\gamma_z \in z$ is geodesic, 
  $(z|\gamma_z(t))_o > |\gamma_z(t)| -
  \delta'$
  and $|\gamma_z(t)| =
  t$.  Likewise,
  since $\rho_z \in z$ is $(\kappa_0,12\delta)$-quasi-geodesic, 
  $(z|\rho_z(t))_{x_0} > |\rho_z(t)|- M - \delta$, where $M =
  M(\kappa_0,12\delta)$ is the Morse constant, and 
  $|\rho_z(t)| = \kappa_z(t)t + O_{\delta}(1)$ with $\frac 1
  {\kappa_0} \leq \kappa_z(t) \leq \kappa_0$. 
  Let $y = (z,t), y' = (z',t') \in
  \Co Z$.  By \cite[Lem.~5.1]{bonk}, we have
  \begin{align*}
    |yy'| &= |\gamma_z(t)\gamma_{z'}(t')| \\
          &= |\gamma_z(t)| + |\gamma_{z'}(t')| - 2\min
  \{|\gamma_z(t)|,|\gamma_{z'}(t')|,(z|z')_{o}\}+ O_{\delta'}(1) \\
                          &= t + t' - 2\min \{t, t', (z|z')_{o}\}
                          + O_{\delta'}(1)
  \end{align*} 
  and similarly, 
  \begin{align*}
    |\iota(y)\iota(y')| &= |\rho_z(t)\rho_{z'}(t')| \\
                        &= \kappa_z(t)t + \kappa_{z'}(t')t' - 
                        2\min \{\kappa_z(t)t,
                        \kappa_{z'}(t')t', (z|z')_{x_0}\}
                        + O_\delta(1).
  \end{align*}
  $\iota$ is a quasi-isometric embedding. 
\end{proof}

Applying the argument in \cite[Prop.~6.5]{bleb}, we obtain
that $\asdim \Co Z \geq \dim Z + 1$.
Proposition~\ref{prop:bdry_cpt} then follows from
Lemma~\ref{lem:cone_embed}. \sqed

\section{Asymptotic dimension of $\G(\Omega)$}\label{sec:ga} 
We prove Theorem~\ref{thm:ga}.  
Let $\Omega$ be a surface of infinite topological type 
with finite grand splitting $\SC(\Omega)$.
\begin{definition}\label{def:fullsep}
  An essential, connected, compact subsurface
$\Sigma \subset \Omega$ is
\textit{fully separating} if 
every component of $\partial \Sigma$ is separating.
\end{definition}

\noindent
Any compact subsurface can be enlarged to one that is fully
separating: \textit{e.g.}\ we may glue $1$-handles between boundary
components adjacent to the same complementary component and take
the compact surface filled by the result.

\begin{lemma}\label{lem:embed}
  Suppose that $\Sigma \subset \Omega$ is a fully separating
  non-annular witness for $\mathcal{G}(\Omega)$ and 
  $|\mathcal{S}(\Omega)| = m$.  There exists a minimally 
  $m$-partite relation $\Gamma$ on $\pi_0(\partial
  \Sigma)$ such that $\A(\Sigma,\Gamma)$
  quasi-isometrically embeds into $\mathcal{G}(\Omega)$.
  In particular, $\Gamma$ 
  is not bipartite if $|\mathcal{S}(\Omega)| > 2$.
\end{lemma}

\begin{proof}
  Since $\Sigma$ is a witness for $\G(\Omega)$, it must separate 
  distinct sets in $\SC(\Omega)$.  In particular, each
  boundary component is adjacent to a complementary component
  containing ends in at most one set in $\SC(\Omega)$.  
  Color each component $c \in \pi_0(\partial \Sigma)$ with the 
  corresponding set $e(c) \in \SC(\Omega)$, if
  one exists; let $\Gamma$ be the complete $m$-partite relation on
  these colors (components without a corresponding 
  class are left isolated).

  Fix a hyperbolic metric on $\Sigma$.  For each colored 
  boundary component $c$, choose a parameterization $c : [0,1) \to
  \Sigma$ and a simple ray $\rho_c$ disjoint from $\inter(\Sigma)$ 
  with origin $c(0)$ and converging to an end in $e(c)$.  
  Let $a \in \A(\Sigma,\Gamma)$ be an arc that terminates on
  $c_1, c_2 \in \pi_0(\partial
  \Sigma)$.  Let $\alpha$ be the geodesic representative for $a$
  with endpoints $c_i(t_i)$ and define $\delta_i = c_i|_{[0,t_i]}$
  to be the subpath of $c_i$ between $c_i(0)$ and $c_i(t_i)$.
  Let $\alpha^\dag$ denote the extension of $\alpha$ from both
  endpoints by $\bar \delta_i * \rho_{c_i}$, 
  for $i = 1,2$ as appropriate. $\alpha^\dag$ 
  is a simple arc converging
  to ends in $e(c_1),e(c_2)$ respectively, which are 
  distinct in $\SC(\Omega)$ by our choice of $\Gamma$.
  $\alpha^\dag$ is a grand arc.  The map $a \mapsto [\alpha^\dag]$
  preserves disjointness hence 
  extends to a simplicial ($1$-Lipschitz) map $\psi : 
  \A(\Sigma,\Gamma) \to \G(\Omega)$. 

  We show that $\psi$ is a quasi-isometric embedding by constructing
  a coarse Lipschitz retraction $\pi : 
  \G(\Omega) \to \A(\Sigma,\Gamma)$.
  For a grand arc $w \in \G(\Omega)$, fix a representative $\omega$ 
  that is geodesic in $\Sigma$. Let $\omega^\pm$ denote the first
  and last intersections of $\omega$ with $\Sigma$ and let 
  $\hat \omega$ denote the shortest path between $\omega^-$ and
  $\omega^+$ in $(\omega \cap \Sigma) \cup \partial \Sigma$.
  Since $\omega$ converges to maximal ends distinguished by
  $\SC(\Omega)$,
  $\omega^\pm$ lie on boundary components with distinct colors:
  isotoping $\hat \omega$ into the interior of $\Sigma$ rel
  $\omega^\pm$, $\hat \omega$ is $\Gamma$-allowed and we define $\pi
  : w \mapsto [\hat \omega]$. From the constructions of $\psi, \pi$,
  it is immediate that $\pi\psi$ is identity on $\A(\Sigma,\Gamma)$.
  We verify that $\pi$ is Lipschitz.  Let $w,
  w' \in \G(\Omega)$ be disjoint grand arcs and let $\pi(w) = [\hat
  \omega]$ and $\pi(w') = [\hat\omega']$ as above.  
  Since $\hat \omega$ is constructed as a shortest
  path,  it contains at most $|\pi_0(\partial \Sigma)| -
1$ segments that are components of $\omega \cap \Sigma$. Each of
these segments intersects $\hat \omega'$ at most twice and  in
subsegments of $\hat \omega'$ parallel to $\partial \Sigma$, and the
same statement holds exchanging $\hat\omega$ and $\hat\omega'$.
Thus $i(\hat \omega,\hat\omega') \leq 4|\pi_0(\partial \Sigma)| -
4$.  Finally, since $d([\hat\omega],[\hat\omega']) \leq
i(\hat\omega,\hat\omega') + 1$ by \cite[Prop.~2.6]{pac}, 
we obtain that $\pi$ is $(4|\pi_0(\partial \Sigma)| - 3)$-Lipschitz.
\end{proof}  

Witnesses for $\G(\Sigma)$ exist 
\cite[Lem.~2.7]{grand} 
and their enlargements are likewise witnesses, hence 
there exist fully separating witnesses
$\Sigma \subset \Omega$ of arbitrarily large complexity.
If $|\SC(\Omega)| > 2$ and $\Gamma$ is chosen as in
Lemma~\ref{lem:embed}, 
then $\A(\Sigma,\Gamma)$ is $\delta$-hyperbolic by
Theorem~\ref{thm:pac_hyp} and by Lemma~\ref{lem:embed} and
Theorem~\ref{thm:asdim} $\asdim \mathcal{G}(\Omega) > n$ for all
$n$. 

Suppose instead that $|\SC(\Omega)| = 2$.  If $\Omega$ has infinite
genus or infinitely many non-maximal ends, then 
there exists an infinite collection  of pairwise-disjoint annular
witnesses  separating the sets 
$\{e,f\} = \SC(\Omega)$.
Choosing finite subcollections defines quasi-flats of arbitrarily
large dimension \cite[Exercise 3.13]{curvenotes}, 
hence again $\asdim \G(\Omega) 
= \infty$.  Alternatively, $\G(\Omega)$
contains an asymphoric
hierarhically hyperbolic space of arbitrarily high rank,
hence has infinite asymptotic dimension \cite[Prop.~1.11]{hhs_gen}.  

Finally, suppose that $|\SC(\Omega)| = 2$ and $\Omega$ has finite
genus and finitely many non-maximal ends.  $\Omega$ must have at
least one infinite set $e \in \SC(\Omega)$; 
let $f \in \SC(\Omega)$ be the other set.  For any $n$, choose a
$(n+1)$-holed sphere $\Sigma  \subset \Omega$ 
with $n$ boundary components partitioning $e$ and the remaining
component separating $e$ from $f$ and any genus or non-maximal ends.
Then $\Sigma$ is a fully separating
witness for $\G(\Omega)$ and $\Gamma$, defined
as in Lemma~\ref{lem:embed}, is a $n$-pointed star.
$\A(\Sigma,\Gamma)$ is $\delta$-hyperbolic by
Theorem~\ref{thm:pac_hyp}: we conclude by Lemma~\ref{lem:embed} and
Theorem~\ref{thm:asdim}. \sqed

\section{Asymptotic dimension of arc and curve
models}\label{sec:cm} 

We generalize the preceding arguments to a broad class of
arc and curve models for finite and infinite-type surfaces.  We
first consider \textit{cocompact} arc and curve models, which we
demonstrate to be equivariantly quasi-isometric to cocompact marking
graphs in the finite-type case; it follows that such arc and
curve models are hierarchically hyperbolic by
\cite{hhs_gen}.  
We then compute asymptotic dimension lower bounds for such models,
which we use to construct lower bounds in the infinite-type
setting.

\subsection{Cocompact arc and curve models}  Let $\Omega$ be a
connected and non-pants hyperbolic 
surface of finite or infinite type.  
We first provide an
extension of arc and curve systems and markings on $\Omega$ 
that subsumes both. Let 
$\mathcal{K}(\Omega) \dfn K(V(\mathcal{AC}(\Omega)))$ denote the set of finite collections of
(not necessarily disjoint) simple arcs and curves on $\Omega$ and
let $\PMap_c(\Omega) \leq \Map(\Omega)$ denote the subgroup of compactly supported
pure mapping classes.  

\begin{definition}\label{def:ccpt}
  An \textit{arc and curve model}  $\M$ on $\Omega$ is a connected 
  simplicial graph with $V(\M) \subset \mathcal{K}(\Omega)$ 
  that admits an action of $\PMap_c(\Omega)$ induced by the
  permutation of its vertices.  $\M$ is \textit{cocompact} if 
  this action is cocompact.
\end{definition}

\begin{remark*}\label{rmk:ccptint}
If $\Omega$ is finite-type, then 
\begin{enumerate*}[label=\textit{(\roman*)}]
  \item $\PMap_c(\Omega) = \PMap(\Omega)$ and
  \item
  $\M$ is cocompact if and only if $i(u,u)$ and $i(u,v)$ are
  uniformly bounded for $u \in V(\M)$ and $(u,v) \in E(\M)$.
\end{enumerate*}
\end{remark*}

Cocompact arc and curve models include most familiar graphs on
finite-type surfaces, such as the arc and curve graph
$\mathcal{AC}(\Sigma)$, connected subgraphs preserved by
$\PMap(\Sigma)$ (\textit{e.g.}\ the
curve graph and arc graph), and the pants graph.  Masur and Minsky's
marking complex $\mathcal{MC}(\Sigma)$ is likewise included, as we will show below.

\subsubsection{Cocompact marking graphs}\label{sec:marking}
A \textit{marking graph} $\mathcal{L}$ on a
compact surface $\Sigma$ is a connected simplicial graph whose
vertices are locally clean markings on $\Sigma$ \cite{hhs_gen}.
As above, $\mathcal{L}$ is \textit{cocompact} if $\PMap(\Sigma)$ acts on
$\mathcal{L}$ cocompactly by permuting its vertices.  We prove that
cocompact arc and curve models and cocompact marking graphs are
identical, up to $\PMap(\Sigma)$-equivariant quasi-isometry.

For an essential subsurface $W \subset \Sigma$ and $u$ a collection
of simple arcs and curves, let $\pi_W(u) \subset \cc W$ denote union
of the (set-valued) projections of the elements in $u$, and let $d_W(u,v)
\dfn \diam( \pi_W(u) \cup \pi_W(v))$.
We show:

\begin{proposition}\label{prop:ggm_to_mg}
  Let $\M$ be a cocompact arc and curve model on a compact surface $\Sigma$.
  Then there exists a cocompact marking graph $\LL_{\M}$ on $\Sigma$
  with an identical witness set and
  an equivariant coarse quasi-isometry $\zeta : \M \to \LL_{\M}$ that
  coarsely preserves projections to witness curve graphs.
\end{proposition}

\begin{lemma}\label{lem:choi-rafi}
  There is a uniform increasing function $\phi : \R_+ \to \R_+$
  depending only on $\Sigma$ such that for any $a,b \in
  \mathcal{AC}(\Sigma)$ and any essential subsurface $W \subset
  \Sigma$, $\phi(i(a,b)) \geq d_W(a,b)$.
\end{lemma}

\begin{proof}
  When $a,b$ are both curves, the claim follows from the Choi-Rafi
  formula for curves \cite[Thm.~2.10]{wata}.  In general, there exists some
  increasing (affine) function $\phi' : \R_+ \to \R_+$ so that
  $\phi'(i(a,b)) >
  i(a',b')$ for every $a' \in \pi_\Sigma(a), b' \in
  \pi_\Sigma(b)$.  We conclude by observing that $\pi_W \circ \pi_\Sigma$ and  
  $\pi_W$ are uniformly boundedly close.
\end{proof}

\begin{remark*}\label{rmk:proj} 
  It follows from Remark~\ref{rmk:ccptint} and Lemma~\ref{lem:choi-rafi}
that for any cocompact arc and curve model $\M$, 
witness subsurface projections $\pi_W : V(\M) \to 2^{\cc W}$ 
are uniformly Lipschitz with uniformly bounded
vertex images; in general, any subsurface projection has uniformly bounded
diameter vertex images.
\end{remark*}

\begin{definition}\label{def:marking}
  A \textit{generalized marking} $\mu = \{(a_i,t_i)\}$ on a surface $\Sigma$ is a
multicurve $\base \mu = \{a_i\}$ along with
bounded subsets $t_i \subset \cc (a_i)$ for each $a_i$.  
Let $\mu' \subset \mu$ be the
maximal submarking with only non-empty transversals.  If $\base
\mu'$ is a pants decomposition on $\Sigma \setminus (\mu \setminus
\mu')$, then $\mu$ is \textit{locally complete}.
\end{definition}

\begin{definition}\label{def:compatible}
  A locally clean marking $\mu' = \{(a_i',t_i')\}$ is \textit{compatible} with a
  generalized marking $\mu = \{(a_i,t_i)\}$ if $a_i' = a_i$ (up to
  reindexing), $t_i = \varnothing$ only if $t_i' = \varnothing$, 
   and $\diam(t_i',t_i)$
  is minimal among all choices of clean $t_i'$.
\end{definition}

\noindent
Only a locally complete generalized marking admits compatible locally
clean markings.
Exactly analogously to \cite[Lem.~2.8]{MMII}, we have the following
for generalized markings:

\begin{lemma}\label{lem:compatible}
For any locally complete generalized marking $\mu = \{(a_i,t_i)\}$,
there is a constant $n_0$ depending only on $\max_i (\diam(t_i))$ 
and at least
one and at most $n_0^{|\mu|}$ compatible locally clean markings.
For any compatible marking $\mu' = \{(a_i,t_i')\}$,
$d_{a_i}(t_i,t_i') < n_1$ for a universal constant $n_1$.
\end{lemma}

 For each $u \in V(\M)$, we first equivariantly construct
  a corresponding locally clean marking $\mu_u$. 
  Choose representatives for $u = \{a_i\}$ in
  pair-wise minimal position; let $\Gamma_u = \bigcup_i a_i$, viewed
  as a $1$-complex (add basepoints to isolated simple curves as necessary).
For a
subcomplex $\Gamma' \subset \Gamma_u$, let $\partial_\Sigma(\Gamma')$
denote the set of essential, non-peripheral boundary components of a regular neighborhood
of $\Gamma' \cup \partial \Sigma$, or equivalently
the set of non-peripheral
boundary components of the essential subsurface filled by
$\Gamma'$.  Fix an enumeration of the edges
in $\Gamma_u$ and let $\Gamma_{u,j} \subset \Gamma_u$
be an exhaustion of $\Gamma_u$ by adding the $j$th successive edge;
let $F_u$ denote the essential subsurface filled
by $u$ and 
$F_{u,j}$ the essential subsurface filled by $\Gamma_{u,j}$.
Let $m_u$ be the multicurve $\bigcup_j
\partial_\Sigma(\Gamma_{u,j})$.  Obtain the generalized
marking $\tilde \mu_u$ by adding to each component $c_i \in m_u$ the
transversal $\pi_{c_i}(u)$.

\begin{claim}\label{claim:muu} 
$\tilde \mu_u$ satisfies the following:  
\begin{enumerate}[label=(\roman*)]
    \item \label{item:muu:complete} The submarking
      $\tilde \mu_u \setminus \partial F_u$ is contained in $F_u$ and
      complete on $F_u$.
    \item \label{item:muu:trans} $\tilde \mu_u \cap \partial F_u$ has empty transversals, and for
      all other $(c_i,t_i) \in \tilde \mu_u$, 
      $d_{c_i}(t_i, \pi_{c_i}(u))$
      is uniformly bounded independently of $u$.
    \item \label{item:muu:int} For $c_i \in \base \tilde \mu_u$, $i(c_i,u)$ is uniformly
      bounded independently of $u$.
  \end{enumerate}
  In particular, $\tilde \mu_u$ is locally complete and any
  compatible locally clean marking satisfies
  \ref{item:muu:complete}-\ref{item:muu:int}.
\end{claim} 

\begin{proof}[Proof of claim]
  Let $e_j = \overline{\Gamma_{u,j}\setminus \Gamma_{u,j-1}}$ 
  be the $j$th edge
in $\Gamma_u$ and let $\Gamma_{u,0} = \varnothing$.  Let $m_{u,k} =
\bigcup_{j=1}^k\partial_\Sigma(\Gamma_{u,j})$
and $m_{u,0} = \varnothing = F_{u,0}$.
Inducting on $k \leq |E(\Gamma_u)|$, we first claim that $m_{u,k}
\setminus \partial F_{u,k}$ is a pants
decomposition for $F_{u,k}$.
In particular, $m_u \setminus \partial F_u$ is a pants decomposition
of $F_u$ and \ref{item:muu:complete} follows: since $u$ fills $F_u$, it
intersects every component $c_i$ of $m_u \setminus \partial F_u$, which
then has transversal $\pi_{c_i}(u) \neq \varnothing$ in $\tilde
\mu_u$.

By construction $m_{u,k}\setminus \partial F_{u,k}$ is
contained in $F_{u,k}$.
The claim holds vacuously for $k = 0$; suppose it holds for
$k-1\geq 0$ and consider the $k$th edge $e_k$.  Recall that
$\partial_\Sigma(\Gamma_{u,j})$ is the set of non-peripheral
boundary curves in $F_{u,j}$, hence if $F_{u,k}$ and $F_{u,k-1}$ are
isotopic then
$m_{u,k} = m_{u,k-1}$ and the claim holds by induction.  Assume that
$F_{u,k} \supsetneq F_{u,k-1}$.
If $e_k$ is disjoint from
$F_{u,k-1}$, then $F_{u,k}$ is obtained from $F_{u,k-1}$ by adding
either a disjoint annulus or a disjoint pair of pants; in either
case, the claim holds.  Else, $F_{u,k-1}$ is obtained by adding a
$1$-handle with core $e_k$, hence by adjoining a pair of pants along
a (non-peripheral) boundary curve $c \in \partial
F_{u,k-1} \cap m_{u,k-1}$: $m_{u,k} \setminus \partial F_{u,k} = (m_{u,k-1}
\setminus \partial F_{u,k-1}) \cup \{c\}$ is a pants decomposition
for $F_{u,k}$. 

To prove \ref{item:muu:trans}, observe that  
by construction $\partial F_u$ does not intersect $u$, hence $\mu_u
\cap \partial F_u$ has empty transversals; the remaining
transversals are uniformly bounded by Remark~\ref{rmk:proj}.
Finally, every
component $c \in m_u$ is a boundary curve for some $F_{u,k}$, which
is filled by $\Gamma_{u,k} \subset \Gamma_u$: any essential
intersection between $c$ and $u$ implies an essential 
self-intersection, whence \ref{item:muu:int} follows.\!\! \end{proof}

\noindent
Let $\mu_u$ be a choice of locally clean marking compatible with
$\tilde \mu_u$.  While $\mu_u$ and $\tilde \mu_u$ are non-canonical, we
enforce that the choices are $\PMap(\Sigma)$-equivariant:
construct $\mu_u$ for representatives $u \in V(\M)/\PMap(\Sigma)$
and declare $\mu_{\varphi u} \dfn \varphi \mu_u$ for $\varphi \in
\PMap(\Sigma)$.

\begin{proof}[Proof of Prop.~\ref{prop:ggm_to_mg}]
    Let $V(\LL_\M) = \{\mu_u : u \in V(\M)\}$ and obtain $E(\LL_\M)$ by pushing
  forward the edge relation on $\M$ via $u \mapsto \mu_u$; 
  let
  $\zeta : \M \to \LL_\M$ be the $\PMap(\Sigma)$-equivariant 
  surjective Lipschitz map induced by $u \mapsto \mu_u$. Since $\M$
  is cocompact, $\zeta$ implies that $\LL_\M$ is cocompact.

  Let $W
\subset \Sigma$ be an essential, non-pants subsurface.  
  Recall that $u$ is filling on 
  $F_u$,  hence
  Claim~\ref*{claim:muu}\ref{item:muu:complete}-\ref{item:muu:trans} implies
  $W$ intersects $\mu_u$ if and only if it intersects $F_u$ if and
  only if it intersects $u$: the witness set
  for $\LL_\M$ is identical to that of $\M$.
  Claim~\ref{claim:muu}\ref{item:muu:trans} implies that the
  projection of $u$ and $\mu_u$ to annular witnesses parallel to
  $\mu_u$ is uniformly close; similarly, by
  Claim~\ref*{claim:muu}\ref{item:muu:int} $i(u,\base \mu_u)$ is
  uniformly bounded, hence Lemma~\ref{lem:choi-rafi} implies that the
  projection of $u, \mu_u$ to all other witnesses is likewise
  uniformly close.

  It remains to show that 
  $\zeta$ is a coarse Lipschitz retraction,
  hence a quasi-isometry.  Since $\LL_\M$ has finitely many
  $\PMap(\Sigma)$-orbits of edges, it suffices to show that the
  fibers of $\zeta$ are uniformly bounded. The following lemma
  completes the proof.
\end{proof}

\begin{lemma}\label{lem:counting_fiber}
  The fibers $E_\mu = \zeta^{-1}(\mu)$ are uniformly bounded over $\mu \in
V(\LL_\M)$.
\end{lemma}

\begin{proof}
  Let $u, v \in E_\mu$, hence $F_u = F_v$ and $\mu \setminus
  \partial F_u$  is complete on $F_u$. It suffices to show that $i(u,v)$ is
  uniformly bounded.  Let $u' \dfn u \setminus \partial F_u,v' \dfn
  v \setminus \partial F_u$ denote the subsets of
  non-boundary elements in $u,v$ respectively.  By
  Claim~\ref*{claim:muu}\ref{item:muu:int} $u', v'$
  intersect each component of $\base \mu$ uniformly many times,
  independently of $u,v$ and $\mu$, 
  hence $u',v'$
  have at most uniformly many components in each pair of pants $P\subset F_u \setminus
  \mu$.  Up to Dehn twists along components in $\base \mu$, these
  components are among finitely many arcs and curves in $P$; by
  Claim~\ref*{claim:muu}\ref{item:muu:trans}
  projections to curves in $\base \mu$ are uniformly
  close to the corresponding transversal in $\mu$ 
  (independently of $u,v,\mu$), hence the order of these twists (and
  hence the intersection number)
  for any two components in $P$ is uniformly bounded.  It follows
  that $i(u,v)$ is uniformly bounded independently of $u,v,\mu$.
\end{proof}

\begin{proposition}\label{prop:biject}
  The map $\M \mapsto \LL_\M$ induces a bijection between the
  $\PMap(\Sigma)$-equivariant quasi-isometry types of cocompact arc
  and curve models
  on $\Sigma$ and of cocompact marking graphs on $\Sigma$.
\end{proposition}

\begin{proof}
  Since $\M$ and $\LL_\M$ are equivariantly quasi-isometric, the
  induced map is well-defined and injective.  We show surjectivity.
Given a locally clean marking $\mu = \{(a_i,t_i)\}$, where $t_i =
\varnothing$ or $\pi_{a_i}(b_i)$ for some (unique) clean transverse curve
$b_i$, let $u_\mu = \{a_i\}\cup \{b_i\}$. 
  Let $\LL$ be a cocompact marking graph on $\Sigma$, and let $\M_\LL$
  be the graph obtained as the (connected) push-forward of $\LL$ by the map $\psi
  : \mu
  \mapsto u_\mu$.  $\psi$ is $\PMap(\Sigma)$-equivariant, hence
  $\M_\LL$ is cocompact.
  As above, $\psi : \LL \to \M_\LL$ is a surjective Lipschitz
  retraction, hence a quasi-isometry: any two markings in the fiber
  $E_u = \{\mu : u = u_\mu\}$ differ by at most uniformly many flip moves,
  each of which increases the intersection number by at most $2$.
  Hence $\LL_{\M_\LL}$ is equivariantly quasi-isometric to $\LL$,
  which suffices.
\end{proof}

\subsubsection{Geometry in the finite-type case}\label{sec:hhs}

By \cite[Thm.~2.10]{hhs_gen}, 
cocompact marking graphs on
finite-type surfaces are hierarchically hyperbolic with respect to
witness subsurface projection, and this
geometry is determined up to equivariant quasi-isometry by the
set of connected witnesses \cite[Thm.~2.12]{hhs_gen}.
Let $\X^\M$ and $\hat \X^\M$ denote the sets of witnesses and connected
witnesses of an arc and curve model $\M$, respectively.  
Applying Proposition~\ref{prop:ggm_to_mg}, we obtain: 
\begin{theorem}\label{thm:acm_hhs}
  Let $\M$ be a cocompact arc and curve model on a finite-type
  hyperbolic, non-pants surface $\Sigma$.
  Then $(\M,\X^\M)$ is
  an asymphoric hierarchically hyperbolic space with respect to
  subsurface projection to witness curve graphs $\pi_W : \M \to
  2^{\cc W}, W \in \X^\M$.
\end{theorem}

A \textit{(connected) witness set} on $\Sigma$ is any collection of
essential
compact (connected) subsurfaces without pants components 
that is closed under enlargement and the action of
$\PMap(\Sigma)$.  From \cite[Thm.~2.12]{hhs_gen} and 
Propositions~\ref{prop:ggm_to_mg} and
\ref{prop:biject} we have:

\begin{theorem}\label{thm:witness_char}
 The map $\M \mapsto \hat \X^\M$ induces a bijection between
 coarsely $\PMap(\Sigma)$-equivariant quasi-isometry types of cocompact arc
 and curve models on $\Sigma$ and connected witness sets on
 $\Sigma$.
\end{theorem}

\subsection{Witness-cocompactness}  In the infinite-type setting, we
require a related ``local'' condition, determined by the existence of
projections to cocompact arc and curve models on each witness
subsurface.  Crucially, these projections admit Lipschitz sections,
whence we will obtain asymptotic dimension lower bounds in
Section~\ref{sec:asdim_bounds}.
\subsubsection{Witness projections}  

Let $\Sigma \subset \Omega$ be a compact, essential, 
and connected
subsurface.  Let $\K(\Omega,\Sigma) \subset
\K(\Omega)$ denote the subset of collections of 
arcs and curves 
intersecting $\Sigma$.  We construct a 
projection $\rho_\Sigma : \K(\Omega,\Sigma) 
\to \K(\Sigma)$ as follows (see \textit{e.g.}\
\cite[\S5.2]{curvenotes}). 
Let $\iota : \Sigma \hookrightarrow \Omega$ be the inclusion map 
and let $p : \Omega_\Sigma \to \Omega$ be the 
covering space associated to $\pi_1(\Sigma) \cong \img \iota_* <
\pi_1(\Omega)$ with Gromov closure $\overline\Omega_\Sigma$. 
Let $\tilde \iota : \Sigma \hookrightarrow
\Omega_\Sigma$ 
be the (unique) lift of $\iota$, and  $\bar \iota$ its inclusion 
into $\overline \Omega_\Sigma$.  Fix any homeomorphism $\sigma :
\overline \Omega_\Sigma \to \Sigma$ that is a homotopy inverse for
$\bar \iota$; 
note that $\sigma$ is unique up to homotopy, hence isotopy.  
\[
\begin{tikzcd}
  & \Omega_\Sigma \rar[hook] \ar[d,"p"]  & {\overline \Omega}_\Sigma
  \ar[dll, bend right=50,"\sigma"'] \\ 
  \Sigma \rar[hook,"\iota"'] \ar[ur,hook,"\tilde \iota"] & \Omega 
                                                         &
\end{tikzcd}
\]
Given
$\omega \in \K(\Omega,\Sigma)$, let
$\rho_\Sigma(\omega)$ be the isotopy class defined by the
closures of 
non-peripheral components of $\sigma p^{-1}(a)$ for all $a \in \omega$.  

One verifies that $\rho_\Sigma(\omega)$ is independent of the 
choice of representative for $\omega$ and $\sigma$.
  Likewise, $\rho_\Sigma$ is independent of the choice of embedding
  of $\Sigma$: if $\iota' : \Sigma \hookrightarrow
\Omega$ is
isotopic to $\iota$, then the lift $\bar\iota\,'$ is
isotopic to $\bar \iota$ and thus a homotopy inverse for $\sigma$. 

The natural action of $\PMap(\Sigma)$ on $\K(\Sigma)$
defines an action of $\Map(\Sigma,\partial \Sigma)
\twoheadrightarrow \PMap(\Sigma)$.  Similarly, 
$\Map(\Sigma,\partial \Sigma) \curvearrowright
\K(\Omega,\Sigma)$ via the homomorphism
$\Map(\Sigma,\partial \Sigma) \to \PMap_c(\Omega)$ obtained by
extending by identity.

\begin{lemma}\label{lem:projequi}
 $\rho_\Sigma: \K(\Omega,\Sigma) \to \K(\Sigma)$
is $\Map(\Sigma,\partial \Sigma)$-equivariant.  
\end{lemma}

\begin{proof} Let
$\varphi_0 \in \Map(\Sigma,\partial \Sigma)$, fixing a
representative.  Let $\varphi \in \PMap_c(\Omega)$ be its extension by
identity; since $\varphi$ is (compactly) supported in $\Sigma$, it
lifts to a quasi-isometry on $\Omega_\Sigma$ that extends to a
homeomorphism $\overline \varphi$ on $\overline \Omega_\Sigma$.
Since $\overline\iota \varphi_0 = \overline\varphi \,\overline\iota$
and $\sigma, \overline\iota$ are homotopy inverses, $\varphi_0
\sigma$ and $\sigma \overline \varphi$ are homotopic and thus
isotopic. For $a \in \omega \in \K(\Omega,\Sigma)$, $\sigma
\overline{p^{-1}(\varphi a)} = \sigma \overline\varphi
\overline{p^{-1}(a)}$ is isotopic to $\varphi_0 \sigma
\overline{p^{-1}(a)}$, whence the claim follows.
\end{proof}

\begin{corollary}\label{cor:pbact}
  Let $\phi \in \PMap(\Sigma)$.  Then there exists $\psi
  \in \PMap(\Omega)$ preserving $\K(\Omega,\Sigma)$ 
  such that for any $\omega \in
  \K(\Omega,\Sigma)$, $\phi\rho_\Sigma(\omega) =
  \rho_\Sigma(\psi \omega)$. \qed
\end{corollary}

Given an arc and curve model $\M$ on $\Omega$,
let $V(\M), E(\M)$ denote its vertex and edge sets,
respectively. 
If $\Sigma$ is a witness for $\M$ then
$V(\M) \subset \K(\Omega,\Sigma)$ and
$\rho_\Sigma$ defines a projection $V(\M) \to \K(\Sigma)$.

\begin{definition}\label{def:admis_cm}
 A connected arc and curve model $\M$ on $\Omega$ 
 is \textit{witness-cocompact} if 
 \begin{enumerate}[label=(\roman*)]
   \item $\M$ admits a (compact) witness,
   \item $\PMap_c(\Omega)$ preserves $V(\M)$ and extends to an action
     on $\M$, and 
   \item for any witness 
     $\Delta \subset \Omega$, there exists
     $L_\Delta$ such that if $(a,b) \in E(\M)$, 
     then $i(\rho_\Delta(a),\rho_\Delta(b)) \leq L_\Delta$. 
 \end{enumerate}
\end{definition}

\begin{remark*}
  When $\Omega$ is finite-type, it deformation retracts to a compact
witness $\overline \Omega$. Since in addition 
$i(\rho_\Delta(a),\rho_\Delta(b)) \leq i(a,b) = i(\rho_{\overline
\Omega}(a),\rho_{\overline \Omega}(b))$, 
(i) is tautological and in (iii) we may choose $L_\Delta =
L_{\overline \Omega}$ to be uniform.
\end{remark*}

\noindent
Witness-cocompact arc and curve models include many graphs of
contemporary interest on infinte-type surfaces, 
including the ray graph, the omnipresent arc
graph, and the grand arc graph.

\subsubsection{Arc and curve models on witnesses} 
Let $\Sigma
\subset \Omega$ be a witness for a witness-cocompact
arc and curve model $\M$ on $\Omega$.  We construct a cocompact 
arc and curve model $\M_\Sigma$ on $\Sigma$ for which the projection
$\rho_\Sigma$ restricts to a Lipschitz map $\M \to \M_\Sigma$, 
along with a Lipschitz coarse section $\iota : \M_\Sigma \to \M$. 
It follows that $\M_\Sigma$ quasi-isometrically embeds into $\M$.

Let $V(\M_\Sigma) = \rho_\Sigma(V(\M))$ and let $(a,b) \in
E(\M_\Sigma)$ if and only if $a \neq b$ and there exist $\tilde a
\in \rho_\Sigma^{-1}(a), \tilde b \in \rho_\Sigma^{-1}(b)$ such that
$(\tilde a, \tilde b) \in E(\M)$.  It is immediate that $\rho_\Sigma
: V(\M) \to V(\M_\Sigma)$ extends to a surjective $1$-Lipschitz map 
$\rho_\Sigma : \M \to \M_\Sigma$, hence in particular 
since $\M$ is connected so is $\M_\Sigma$.  
Likewise, since $\M$ satisfies Definition~\ref{def:admis_cm}(iii),
so does $\M_\Sigma$ for uniform $L = L_{\Sigma}$.  
By Corollary~\ref{cor:pbact} 
$\PMap(\Sigma)$ acts naturally on $\M_\Sigma$, hence $\M_\Sigma$ is
a cocompact arc and curve model on $\Sigma$.

Fix any $\Map(\Sigma,\partial \Sigma)$-equivariant 
section $\iota : V(\M_\Sigma) \to V(\M)$, and let 
$\tilde a = \iota(a) \in \rho_{\Sigma}^{-1}(a)$.  We show that
$\iota$ is Lipschitz, hence extends to a Lipschitz coarse section
$\iota : \M_\Sigma \to \M$ for $\rho_\Sigma$. 
Since for any $(a,b) \in E(\M_\Sigma)$, $i(a,b) \leq L$, there are
finitely many $\Map(\Sigma,\partial \Sigma)$-orbits of edges in
$\M_\Sigma$.  Let \[
  M = \max_{(a,b) \in E(\M_\Sigma)/G} d_{\M}(\tilde a,\tilde b)
\]
where $G = \Map(\Sigma,\partial \Sigma)$.
Then $\iota$ is $M$-Lipschitz. 
We have shown: 

\begin{proposition}\label{prop:cm_embed}
  Let $\Sigma \subset \Omega$ be a 
  witness for a witness-cocompact arc and curve
  model $\M$ on $\Omega$.  There exists a cocompact arc and curve
  model $\M_\Sigma$ on $\Sigma$ which quasi-isometrically embeds
  into $\M$. \qed
\end{proposition}

\begin{remark*}\label{rmk:wccpt_alt}
  A connected arc and curve model $\M$ on $\Omega$ 
  with a natural action of $\PMap_c(\Omega)$ 
  is witness-cocompact if and only if it has a witness and 
  the projection $\M_\Delta$ is cocompact for every witness 
  $\Delta\subset \Omega$.
\end{remark*}

\subsection{Asymptotic dimension lower
bounds}\label{sec:asdim_bounds}   We first  
consider the asymptotic dimension of cocompact arc and curve
models on $\Sigma$, a finite-type non-pants hyperbolic surface.  
Up to deformation retraction, we assume $\Sigma$ is compact.  

\begin{remark*}\label{rmk:torus} 
If $\Sigma$ is a (closed) torus, then any cocompact arc and curve
model is quasi-isometric to the curve graph, hence a quasi-tree with
$\asdim = 1$.  Otherwise, 
if $\Sigma$ admits a non-empty cocompact arc and curve
model  (and in particular, a witness subsurface), 
then $\chi(\Sigma) \leq -1$ and $\Sigma \not \cong \Sigma_0^3$. 
\end{remark*}

Let $\M$ be a cocompact arc and curve model on $\Sigma$; 
 recall that
$(\M,\X^\M)$ is an asymphoric hierarchically hyperbolic space by
Theorem~\ref{thm:acm_hhs}.
Then in particular the \textit{rank} $\nu$ 
of $(\M,\X^\M)$ corresponds to the
largest cardinality of a set of pairwise disjoint, connected
witnesses in $\X^\M$. Since $(\M,\X^\M)$ is asymphoric, $\asdim \M \geq
\nu$ \cite[Thm.~1.15]{quasiflats} and $\M$ is $\delta$-hyperbolic if
and only if $\nu = 1$ \cite[Cor.~2.15]{quasiflats}. The lower bound
here will prove sufficient except when $\nu = 1$; we note that an
identical bound can be achieved by explicitly constructing
quasi-flats.

\subsubsection{The $\delta$-hyperbolic case}
Adapting the arguments in Section~\ref{sec:asdim}, we prove the
following:
\begin{theorem}\label{thm:cm_hyp}
  Let $\Sigma$ be a genus $g$ compact surface, 
  possibly with boundary.
  If $\M$ is a (non-empty) $\delta$-hyperbolic cocompact
  arc and curve model on $\Sigma$, then $\asdim \M \geq g -
  \lceil \frac 12 \chi(\Sigma) \rceil$. 
\end{theorem}

\noindent
If $\Sigma \cong \Sigma_1$, then the claim is immediate by
Remark~\ref{rmk:torus}.  Otherwise,
we may assume $\M$ is a cocompact marking graph
by Proposition~\ref{prop:ggm_to_mg}.  
For any $\M'$ a cocompact marking
graph on $\Sigma$ with $\X^\M \supset \X^{\M'}$, there exists a
functorial 
canonical coarse surjection $\iota : \M \to \M'$ such that $\pi_W
\circ \iota$ is uniformly 
coarsely $\pi_W$ for any $W \in \X^{\M'}$ \cite[\S 2.1]{hhs_gen}. 
In particular, $\X^{\mathcal{MC}(\Sigma)}$ is every essential,
non-peripheral subsurface in $\Sigma$ and $\X^{\cc \Sigma} =
\{\Sigma\}$, hence we have canonical maps
$\mathcal{MC}(\Sigma) \to \M \to \cc \Sigma$.

\begin{lemma}\label{lem:canon_align}
  Let $\M, \M'$ be cocompact marking graphs on $\Sigma$, a compact
  surface, such that $\X^\M \supset \X^{\M'}$, and let $\iota : \M
  \to \M'$ be the canonical coarse surjection.  If $\M$ is
  $\delta$-hyperbolic, then $\iota$ is
  coarsely alignment-preserving.
\end{lemma}
\noindent
We note that if $\M$ is $\delta$-hyperbolic, then $\nu(\M') \leq
\nu(\M) \leq 1$, hence $\M'$ is $\delta$-hyperbolic.  Recall that 
a path $\rho \subset X$ is a $D$-\textit{hierarchy path} for a
hierarchically hyperbolic space $(X,\mathscr{G})$ if it is a
$(D,D)$-quasi-geodesic and $\pi_\alpha\rho$ is a unparameterized
$(D,D)$-quasi-geodesic for all $\alpha \in \mathscr{G}$.

\begin{proof}
  Since $(\M,\X^\M)$ is hierarchically hyperbolic, 
  there exists $D > 0$ such that for any $x,y \in \M$, there exists
  a $D$-hierarchy path joining $x,y$ \cite[Thm.~4.4]{hhsii}.  
  Let $(x,z,y) \in \M^3$ be aligned and
  let $\gamma$ be the geodesic from $x$ to $y$ passing through $z$
  and $\rho$ the hierarchy path between $x,y$. By the Morse
  lemma, there exists a constant $M(D,\delta)$ 
   such that $\gamma,\rho$ are $M(D,\delta)$-Hausdorff close, 
   hence $d(z,\rho) \leq M(D,\delta)$.  
   For any $W \in \X^\M$, $\pi_W
   \rho$ is an unparameterized $(D,D)$-quasi-geodesic. 
   Applying the Morse
   lemma and that $\pi_W$ is $L$-Lipschitz for uniform $L$, it
   follows
   that $(\pi_W(x), \pi_W (z), \pi_W(y))$ are $K$-aligned where 
   $K = 2(M(D,\delta_0) + LM(D,\delta))$ 
   is uniform over $\M^3, \X^\M$ and 
   $\delta_0$ is a uniform hyperbolicity constant for curve
   graphs \cite{hpw}. 

   It follows that $\pi_W$ for $W \in \X^\M$ and 
   likewise $\pi_{W'}$ for $W'
   \in \X^{\M'}$ are $K'$-alignment preserving for uniform $K'$.
   Since $\X^{\M'} \subset \X^\M$, the distance formulas for
   $\M,\M'$ imply the claim.
\end{proof}

Suppose that $\M$ is a $\delta$-hyperbolic cocompact 
marking graph on a compact surface $\Sigma$ with genus $g$.
By Lemma~\ref{lem:canon_align}, the canonical map $\iota : \M \to
\cc\Sigma$ is coarsely alignment preserving, hence by
Theorem~\ref{thm:dtboundaries} $\partial \cc \Sigma $
 embeds into $\partial \M$. 
 To prove Theorem~\ref{thm:cm_hyp} it suffices to find a compact 
 subspace $Z \subset \partial \cc
 \Sigma$ such that $\dim Z \geq n \dfn g - 
1 -  \lceil \frac 12 \chi(\Sigma) \rceil$, since  by
 Proposition~\ref{prop:bdry_cpt} $\dim Z + 1 
 \leq \asdim \M$. Recall that 
 $\partial \cc \Sigma \cong \mathcal{EL}(\Sigma)$.  

 \begin{proposition}\label{prop:elSembed}
   Let $\Sigma$ be a genus $g$ compact hyperbolic surface and $S$
   the $(n+4)$-times punctured sphere, where $n = g - 1 -
  \lceil \frac 12 \chi(\Sigma) \rceil$. Then
   $\el(S)$ embeds into $\partial \cc \Sigma \cong \el(\Sigma)$.
 \end{proposition}

\begin{proof}
  For simplicity, we replace 
   the boundary components of $\Sigma$ with punctures, 
   noting that $\cc\Sigma \cong
   \cc(\Sigma \setminus \partial \Sigma)$.
Choose a hyperelliptic involution $\eta$ on $\Sigma$ 
that fixes at most one puncture and let $h : \Sigma \to S'$ be the
orbifold covering map obtained by quotienting by $\eta$.
Obtain $S$ by removing the cone points of $S'$: one verifies that
$S$ has $n + 4$ punctures.  By \cite{orbi}, $h$ induces a
quasi-isometric embedding $h_* : \cc S \to \cc \Sigma$, which has  
quasi-convex image by the Morse lemma. 
Hence $\el(S) \cong \partial \cc S \subset \partial \cc \Sigma$.  
\end{proof}


When $\Sigma$ is a sphere with four boundary components,
Theorem~\ref{thm:cm_hyp} is vacuously true.  Otherwise, 
from Theorem~\ref{thm:gabai} and the universal
embedding property of N\"obeling spaces, we obtain the desired
subspace $Z \subset \el(S) \subset \partial \cc
\Sigma$ and Theorem~\ref{thm:cm_hyp} follows. \sqed

\subsubsection{Lower bounds for infinite-type surfaces}
Given a witness-cocompact
arc and curve model $\M$ on an infinite-type surface
$\Omega$, let $w_\M \in \N\cup \{\infty\}$ denote the least upper
bound on cardinalities for  a set of pairwise-disjoint 
connected witnesses for $\M$. We consider two
cases:

\begin{enumerate}[label=\textit{(\roman*)}]
  \item \textit{$w_\M$ is infinite.} 
    For arbitrarily large $m \in \N$, we may choose a compact,
    essential subsurface $\Sigma \subset \Omega$ containing at least
    $m$ disjoint witnesses.  $\Sigma$ is a witness for $\M$, and any
    witness for $\M$ contained in $\Sigma$ is a witness for
    $\M_\Sigma$ by construction.  It follows that
    $\M_\Sigma$ is an asymphoric hierarchically
    hyperbolic space of rank $\nu \geq m$, hence by
    Proposition~\ref{prop:cm_embed} $\asdim \M \geq
    \asdim \M_\Sigma \geq m$.  $\asdim \M = \infty$.

  \item \textit{$w_\M = m$ is finite.}  Fix a collection of
    pairwise disjoint witnesses $\{W_i\}$ with cardinality $m$.  
    Fix $W_0$ among these such that $W_0$ lies in a complementary
    component $\Omega_0$ of $\bigcup_{i > 0} W_i$ of infinite type. 
    Let $\Sigma \subset \Omega_0$ be an enlargement of $W_0$ of
    arbitrarily negative $\chi(\Sigma)$:  
    $\Sigma$ is a witness for $\M$.
    Moreover, since any witness for $\M_\Sigma$ is a witness for
    $\M$ disjoint from the $W_{i > 0}$, any two connected 
    witnesses for
    $\M_\Sigma$ must intersect: $\M_\Sigma$ is an asymphoric
    hierarchically hyperbolic space of rank $\nu = 1$, hence
    $\delta$-hyperbolic.  By Proposition~\ref{prop:cm_embed} and 
    Theorem~\ref{thm:cm_hyp}, $\asdim \M \geq \asdim \M_\Sigma \geq
    - \frac 12 \chi(\Sigma)$, hence $\asdim \M = \infty$.
\end{enumerate}

\begin{theorem}\label{thm:cm_lower}
  Let $\M$ be a witness-cocompact arc and curve model on an infinite-type
  surface $\Omega$.  Then $\asdim \M = \infty$. \qed
\end{theorem}

\noindent
Theorem~\ref{thm:ma_asdim} follows from Theorems~\ref{thm:cm_hyp}
and \ref{thm:cm_lower}.  \sqed 

\section{Acknowledgements}

The author would like to thank Mladen Bestvina for his support and
advice, George Shaji for numerous helpful discussions, Yusen Long
for comments on the manuscript, and Priyam Patel for suggesting that
the results in Sections~\ref{sec:asdim} and \ref{sec:ga} be
generalized.
The author was supported by NSF awards no.\ 2304774 and 
no.\ 1840190: \textit{RTG: Algebra, Geometry, and Topology at 
  the University of Utah}. 

\bibliographystyle{amsalpha}
\bibliography{pac}


\end{document}